\documentclass[a4paper]{amsart}
\usepackage{esint,latexsym,amsmath,amssymb,amsthm}
\usepackage[abbrev,backrefs]{amsrefs}
\usepackage{a4wide}
\usepackage{tocvsec2}
\usepackage{enumerate}
\usepackage{wasysym}
\usepackage{pifont}
\usepackage{mathrsfs}
\usepackage{CJK}

\usepackage{tikz}
\usepackage{mathdots}
\usepackage{yhmath}
\usepackage{cancel}
\usepackage{color}
\usepackage{siunitx}
\usepackage{array}
\usepackage{multirow}
\usepackage{gensymb}
\usepackage{tabularx}
\usepackage{booktabs}
\usetikzlibrary{fadings}
\usetikzlibrary{patterns}

\newtheorem{theorem}{Theorem}[section]
\newtheorem{corollary}{Corollary}[section]

\newtheorem{lemma}{Lemma}[section]
\newtheorem{proposition}{Proposition}[section]

\theoremstyle{remark}
\newtheorem{remark}{Remark}[section]

\theoremstyle{definition}

\numberwithin{equation}{section}

\newcommand{\bbr}{\mathbb{R}}

\newcommand{\R}{{\mathbb R}}

\renewcommand{\I}{\mathcal{I}}

\renewcommand{\P}{\mathcal{P}}
\newcommand{\eps}{{\varepsilon}}

\begin{document}
\title[Asymptotic profile for a Schr\"{o}dinger-Poisson-Slater equation]
{Asymptotic profile of ground states\\ for the Schr\"{o}dinger-Poisson-Slater equation}%

\author{Zeng Liu}
\address{Department of Mathematics\\ Suzhou University of Science and Technology\\ Suzhou 215009\\ P.R. China}
\email{zliu@mail.usts.edu.cn (Zeng Liu)}

\author{Vitaly Moroz}
\address{Department of Mathematics\\ Swansea University\\
Swansea SA1~8EN\\ Wales, United Kingdom}	
\email{v.moroz@swansea.ac.uk}



\date{\today}

\begin{abstract}
We study the Schr\"{o}dinger-Poisson-Slater equation
$$-\Delta u + u+\lambda(I_{2}*|u|^2)u=|u|^{p-2}u\quad\text{in $\R^3$},$$
where $p\in (3,6)$ and $\lambda>0$.  By using direct variational analysis based on the comparison of the ground state energy levels, we obtain a characterization of the limit profile of the positive ground states for $\lambda\to \infty$.
\end{abstract}

\maketitle

\setcounter{tocdepth}{2}

\section{Introduction}

We are concerned with the asymptotic profiles of positive ground state solutions of a Schr\"{o}dinger-Poisson-Slater equation
\begin{equation}\tag{$P_\lambda$}\label{eqPlam}
 -\Delta u + u+\lambda(I_{2}*|u|^2)u=|u|^{p-2}u\quad\text{in $\R^3$},
\end{equation}
where $p>2$ and $\lambda> 0$. Here $I_2(x):=(4\pi|x|)^{-1}$ is the Riesz potential and $*$ denotes the standard convolution in $\R^3$.
By a {\em ground state} solution of \eqref{eqPlam} we understand a
weak solution $u_0 \in H^1(\R^3)\setminus\{0\}$ which has a minimal energy amongst all nontrivial solutions of \eqref{eqPlam}, namely $\mathcal{I}_{\lambda}(u_0)\leq \mathcal{I}_{\lambda}(u)$ for any solution $u$ of \eqref{eqPlam},
where $\mathcal{I}_{\lambda}: H^1(\R^3)\to \R$ is the corresponding functional of \eqref{eqPlam} defined as
$$
\mathcal{I}_{\lambda}(u)=\frac{1}{2}\int_{\R^3}|\nabla u|^2dx+\frac{1}{2}\int_{\R^3}|u|^2dx+\frac{\lambda}{4}\int_{\R^3}(I_{2}*|u|^2)|u|^2dx-\frac{1}{p}\int_{\R^3}|u|^p dx,\ \ u\in H^1(\R^3).
$$

Equation~\eqref{eqPlam} appears in quantum mechanics as an approximation of the Hartree-Fock model of a quantum many-body system of electrons \citelist{\cite{BLS03}\cite{L84}\cite{Catto}}, and in semi-conductor theory \cite{BF98} (under the name of the Schr\"odinger-Maxwell equation). From a mathematical point of view, as pointed
out in \cite{R10}, this model presents a combination of repulsive
forces (given by the nonlocal term) and attractive forces (given by
the nonlinearity). The interaction between the two forces gives rise to
unexpected situations concerning the existence, non-existence and
multiplicity of solutions, and their qualitative behavior, see e.g.
\citelist{\cite{A08}\cite{AR08}\cite{AP08}\cite{CFPT20}\cite{CM16}\cite{D02}\cite{DM04}\cite{JZ11}\cite{LLS17}\cite{LM20}\cite{LWZ16}\cite{LZH19}\cite{Mercuri}\cite{SW20}\cite{R06}\cite{R10}\cite{WZ07}} and the
references therein.

It is shown in \cite{R06} that for $p\in (2,3)$  equation \eqref{eqPlam} has no solutions for $\lambda\ge 1/4$ and has at least two positive radial solutions for small $\lambda>0$. One of the solutions is a ground state and a local minimizer, another is a higher energy mountain pass type solution. For $p\in (3,6)$ equation \eqref{eqPlam} has a positive radial ground state for all $\lambda>0$ \citelist{\cite{R06}\cite{AP08}}. For $p=3$ there is at least one radial positive solution for small $\lambda>0$ and no positive solutions for $\lambda\ge 1/4$ \cite{R06}.

Our goal in this work is to describe the asymptotic profile of the ground state solutions of \eqref{eqPlam} when $p\in(3,6)$ and $\lambda\to \infty$. Observe that for $p\neq 3$ the rescaling
\begin{equation}\label{eqResc}
v(x)=\lambda^{\frac{1}{2(3-p)}}u\big(\lambda^{\frac{p-2}{4(3-p)}}x\big)
\end{equation}
transforms \eqref{eqPlam} into the equation
\begin{equation}\label{eqsps}
-\Delta v+\lambda^{\frac{p-2}{4(3-p)}}v+(I_2*|v|^2)v=|v|^{p-2}v\quad\text{in $\R^3$}.
\end{equation}
Then ``zero mass'' equation
\begin{equation}\label{eqssps}
-\Delta v+(I_2*|v|^2)v=|v|^{p-2}v\quad\text{in $\R^3$}
\end{equation}
becomes the {\em formal} limit equation for \eqref{eqsps} either if $p<3$ and $\lambda\to 0$, or $p>3$ and $\lambda\to \infty$. Denote by
$$
E(\bbr^3)=\{u\in D^{1,2}(\bbr^3): \int_{\bbr^3}(I_2*|u|^2)|u|^2dx<\infty\}
$$
the Coulomb-Sobolev space \citelist{\cite{L84}\cite{R10}} and by $E_r(\bbr^3)$ the subspace of radial functions in $E(\bbr^3)$. The space $E(\bbr^3)$ is the natural domain for the energy that corresponds to  \eqref{eqssps}. It is proved in \cite{R10,IR12} that equation \eqref{eqssps}
has a positive radial solution in the space $E_r(\R^3)$ (see below)
for $p\in (18/7, 3)\cup(3, 6)$. Such a radial solution is a global minimizer in $E_r(\R^3)$ for $p\in (18/7, 3)$ and a ground state in $E_r(\R^3)$ for $p\in(3,6)$.
Equation \eqref{eqssps} also has a positive ground state in $E(\R^3)$ (see below) for $p\in (3, 6)$, see \cite{R10}. It is an open problem whether the ground state in $E(\R^3)$ is in fact radial (and hence the two ground states are the same) for $p\in(3,6)$ or for a subset of the interval $(3,6)$, see \citelist{\cite{R10}\cite{MMS16}}.

The asymptotic behavior of the solutions for \eqref{eqPlam} were studied in \cite{DW05,R05,R10} as $\lambda\to 0$. In \cite{DW05,R05}, by using a Lyapunov-Schmidt type perturbation argument, the authors prove that for $p\in (2, 18/7)$ there exist a family of positive radially symmetric bound states of \eqref{eqPlam} that concentrates around a sphere as $\lambda\to 0$.  In \cite{R10} the author proves that when $p\in (18/7, 3)$ the rescaled family of the radial ground states $v_{\lambda}=\lambda^{\frac{1}{2(3-p)}}u_{\lambda}\big(\lambda^{\frac{p-2}{4(3-p)}}x\big)$ of \eqref{eqPlam} converges as $\lambda\to 0$ in $E_r(\R^3)$ to a positive radial ground state (global minimizer) $v_0\in E_r(\R^3)$ of \eqref{eqssps}. To our best knowledge, we are not aware of any results on the asymptotic behaviour of ground state solutions of \eqref{eqPlam} when $p\in (3, 6)$ and $\lambda\to \infty$.

The energy functional $\mathcal{I}_\lambda$ transforms after the rescaling \eqref{eqResc} to the energy
\begin{equation*}
\widetilde{\mathcal{I}}_{\lambda}(v):=\frac{1}{2}\int_{\R^3}|\nabla v|^2dx+\frac{\lambda^{\frac{p-2}{4(3-p)}}}{2}\int_{\R^3}|v|^2dx+\frac{1}{4}\int_{\R^3}(I_{2}*|v|^2)|v|^2dx-\frac{1}{p}\int_{\R^3}|v|^p dx.
\end{equation*}
We denote by
\begin{equation*}
	\widetilde{\mathcal{I}}_{\infty}(v):=\frac{1}{2}\int_{\R^3}|\nabla v|^2dx+\frac{1}{4}\int_{\R^3}(I_{2}*|v|^2)|v|^2dx-\frac{1}{p}\int_{\R^3}|v|^p dx
\end{equation*}
the energy that corresponds to \eqref{eqssps}.
If $p\in(3, 6)$, {\em formally} we have $\widetilde{\mathcal{I}}_{\lambda}\to \widetilde{\mathcal{I}}_{\infty}$ as $\lambda\to\infty$. Observe however that when $p\in(3, 6)$ the energy $\widetilde{\I}_\infty$ is well posed in the space $E(\R^3)$ (cf. \cite{R10,IR12}), while $\widetilde{\I}_\lambda$ with $\lambda>0$ is well-posed in $H^1(\R^3)$. Since $E(\R^3)\subsetneq H^1(\R^3) $, small perturbation arguments in the spirit of the Lyapunov-Schmidt reduction are not directly applicable to the family $\widetilde{\I}_\lambda$ in the limit $\lambda\to \infty$. Using direct variational analysis based on the comparison of the ground state energy levels for two problems, we establish the following result.

\begin{theorem}\label{thm01}
	Let $3<p<6$.  Then for any sequence $\{\lambda_n\}$ with $\lambda_n\to\infty$ as $n\to \infty$, there exists $\{\xi_{\lambda_n}\}\subset \R^3$ such that the rescaled family of ground states of $(P_\lambda)$
	$$
	v_{\lambda_n}(x):={\lambda_n}^{\frac{1}{2(3-p)}}u_{{\lambda_n}}({\lambda_n}^{\frac{p-2}{4(3-p)}}(x+\xi_{\lambda_n}))
	$$
	converges in $E(\R^3)$ to a positive ground state solution $\overline{v}_\infty$ of the formal limit equation \eqref{eqssps}.
	Moreover, $\lambda^{\frac{p-2}{4(3-p)}}\|v_\lambda\|_2^2\to 0$ as $\lambda\to\infty$.
\end{theorem}

We remark that our strategy is quite different form \cite{R10}, and follows from the ideas of \cite{MM14,LM21}. In \cite{R10} the author studied the profile of the radial global minimizer for $p\in (18/7, 3)$ as $\lambda\to 0$. Since for $p\in (18/7, 3)$ the energy functional $\widetilde{\mathcal{I}}_{\infty}$ is coercive, the approach of \cite{R10} is not applicable for $p\in (3, 6)$. Note that Theorem~\ref{thm01} remains valid also for the radial ground state solutions since we can work in the radially symmetric settings step by step.

Our results do not rely and do not require the uniqueness or non-degeneracy of the ground-states of \eqref{eqssps}.  We also point out that since $E(\R^3)\subsetneq H^1(\R^3)$, it is crucial to know that the ground state solution $v_\infty$ of the formal limit equation \eqref{eqssps} has exponential decay at infinity \cite[Theorem 1.3]{BJL13}, and hence belongs to $L^2(\R^3)$. This plays an essential role in the comparison of ground state energy levels between \eqref{eqsps} and \eqref{eqssps}.

\begin{remark}
Our method can be adapted to show that for $3<p<6$ and for any sequence $\{\lambda_n\}$ with $\lambda_n\to 0$ as $n\to \infty$, there exists $\{\xi_{\lambda_n}\}\subset \R^3$ such that the translated family of ground states $u_{{\lambda_n}}(x+\xi_{\lambda_n})$ of $(P_\lambda)$ converges in $H^1(\R^3)$ to the unique positive ground state solution $u_0$ of the local equation
$$-\Delta u + u=|u|^{p-2}u\quad\text{in $\R^3$}.$$
However, since both $(P_\lambda)$ and the local limit problem are well-posed in $H^1(\R^3)$ this also can be seen via small perturbation arguments.
\end{remark}

\section{Preliminaries}

We begin by describing some properties of the space $E(\R^3)$.  The following two propositions have been proved in \cite[Propositions 2.2, 2.4]{R10} and will be used throughout the paper.
\begin{proposition}
Let us define, for any $u\in E(\R^3)$,
$$
\|u\|_E=\Big(\int_{\R^3}|\nabla u|^2dx+\Big(\int_{\R^3}(I_{2}*|u|^2)|u|^2dx\Big)^{\frac12}\Big)^{\frac12}.
$$
Then, $\|\cdot\|_E$ is a norm, and $(E(\R^3), \|\cdot\|_E)$ is a uniformly convex Banach space. Moreover, $C^{\infty}_0(\R^3)$ is
dense in $E(\R^3)$ and also $C^{\infty}_0(\R^3)$ is dense in $E_r(\R^3)$.
\end{proposition}
Let us define $\phi_u=I_2*|u|^2$, then $u\in E(\R^3)$ if and only if $u$ and $\phi_u$ belong to $D^{1,2}(\R^3)$. The characterizations of the convergences in $E(\R^3)$ is given by the the following proposition.
\begin{proposition}\label{pro1203}
Given a sequence $\{u_n\}$ in $E(\R^3)$, $u_n\to u$ in $E(\R^3)$ if and only if $u_n\to u$ and $\phi_{u_n}\to\phi_{u}$ in $D^{1,2}(\R^3)$.

Moreover, $u_n\rightharpoonup u$ in $E(\R^3)$ if and only if $u_n\rightharpoonup u$ in $D^{1,2}(\R^3)$ and $\int_{\R^3}(I_2*|u|^2)|u|^2dx$ is bounded.  In such case,
$\phi_{u_n}\rightharpoonup\phi_{u}$ in $D^{1,2}(\R^3)$.
\end{proposition}
It is also proved in \cite[Theorem 1.2]{R10} that $E(\R^3)\hookrightarrow L^q(\R^3)$ continuously for $q\in [3, 6]$, $E_r(\R^3)\hookrightarrow L^q(\R^3)$ continuously for $q\in (18/7, 6]$, and the inclusion is compact for $q\in (18/7, 6)$.  As in \cite{IR12}, we define $M:E(\R^3)\to R$ as
$$
M(u)=\int_{\R^3}|\nabla u|^2dx+\int_{\R^3}(I_{2}*|u|^2)|u|^2dx,
$$
then we can easily check that for any $u\in E(\R^3)$,
$$
\frac{1}{2}\|u\|^4_E\leq M(u)\leq \|u\|^2_E, \quad \text{if either} \ \|u\|_E\leq 1 \quad \text{or} \ M(u)\leq 1.
$$
The following estimate on $M(u)$ is given by \cite[Lemma 3.1]{IR12}.
\begin{lemma}\label{lem1201}
	Assume that $p\in (3, 6)$. 	Then there exists $C>0$ such that $\|u\|^p_p\leq CM(u)^{\frac{2p-3}{3}}$ for all $u\in E$.
\end{lemma}

To finish this section, we state a Poho\v zaev type identity, see \cite{R06,IR12}.
\begin{proposition}\label{pro1201}
	Assume that $p\in (2, 6)$. 	
\begin{itemize}
\item [(1)]
Let $u\in H^1(\R^3)$ be a weak solution of \eqref{eqsps}, then
$$
\frac{1}{2}\int_{\R^3}|\nabla u|^2dx+\frac{3\lambda^{\frac{p-2}{4(3-p)}} }{2}\int_{\R^3}|u|^2dx+\frac{5}{4}\int_{\R^3}(I_{2}*|u|^2)|u|^2dx-\frac{3}{p}\int_{\R^3}|u|^pdx=0
$$
\item [(2)]
Let $u\in E(\R^3)\cap H^2_{loc}(\R^3)$ be a weak solution of \eqref{eqssps}, then
$$
\frac{1}{2}\int_{\R^3}|\nabla u|^2dx+\frac{5}{4}\int_{\R^3}(I_{2}*|u|^2)|u|^2dx-\frac{3}{p}\int_{\R^3}|u|^pdx=0
$$
\end{itemize}
\end{proposition}

\section{Asymptotic profiles of the rescaled ground state}

It is well known that in \cite{AP08} the authors has obtained the existence of ground states $u_{\lambda}$ of \eqref{eqPlam} when $p\in (3,6)$, which is a mountain pass type solution.  Then the rescaling
$$
v_{\lambda}(x)=\lambda^{\frac{1}{2(3-p)}}u_{\lambda}(\lambda^{\frac{p-2}{4(3-p)}}x)
$$
is a ground state of \eqref{eqsps} and corresponding to the rescaled minimization problem
\begin{equation}\label{eq-energy}
m_\lambda=\inf_{u\in \mathscr{P}_{\lambda}}\widetilde{\mathcal{I}}_{\lambda}(u), \quad \mathscr{P}_{\lambda}:=\{u\in H^1(\bbr^3)\setminus\{0\}: \mathcal{P}_{\lambda}(u)=0\},
\end{equation}
where $\mathcal{P}_{\lambda}:H^1(\bbr^3)\to \R$ is defined by
$$
\mathcal{P}_{\lambda}(u)=\frac{3}{2}\int_{\R^3}|\nabla u|^2dx+\frac{\lambda^{\frac{p-2}{4(3-p)}} }{2}\int_{\R^3}|u|^2dx+\frac{3}{4}\int_{\R^3}(I_{2}*|u|^2)|u|^2dx-\frac{2p-3}{p}\int_{\R^3}|u|^pdx.
$$
Observe that if $u\in H^1(\R^3)$ is a critical point of $\widetilde{\mathcal{I}}_{\lambda}$, then $u\in \mathscr{P}_{\lambda}$ since $\mathcal{P}_{\lambda}(u)=0$ is nothing but the combination of $\langle\widetilde{\mathcal{I}}'_{\lambda}(u),u\rangle=0$ and the Poho\v zaev type identity.
For each $u\in H^1(\bbr^3)\setminus\{0\}$, set
\begin{equation}\label{e-ut}
u_t(x):=t^2u(tx).
\end{equation}
Then
\begin{multline}\label{Peps-non0}
f_u(t):=\widetilde{\mathcal{I}}_{\lambda}(u_t)\\
=\frac{t^{3}}{2}\int_{\R^3}|\nabla u|^2dx+\frac{\lambda^{\frac{p-2}{4(3-p)}} t}{2}\int_{\R^3}|u|^2dx+\frac{t^3}{4}\int_{\R^3}(I_{\alpha}*|u|^2)|u|^2dx-\frac{t^{2p-3}}{p}\int_{\R^3}|u|^pdx.
\end{multline}
Clearly, there exists a unique $t_u>0$ such that $f_u(t_u)=\max\{f_u(t): t>0\}$ and $f'_u(t_u)t_u=0$, which means that $t_u^2u(t_ux)\in \mathscr{P}_{\lambda}$.
Therefore $\mathscr{P}_{\lambda}\neq \emptyset$.  Clearly, $v_{\lambda}\in \mathscr{P}_{\lambda}$ and $\widetilde{\mathcal{I}}_{\lambda}(v_{\lambda})=m_{\lambda}>0$.

While for the equation \eqref{eqssps}, we formally define
\begin{equation}\label{eq-energy-ssps}
m_\infty=\inf_{u\in \mathscr{P}_{\infty}}\widetilde{\mathcal{I}}_{\infty}(u), \quad \mathscr{P}_{\infty}:=\{u\in E(\bbr^3)\setminus\{0\}: \mathcal{P}_{\infty}(u)=0\},
\end{equation}
where $\mathcal{P}_{\infty}:E(\bbr^3)\to \R$ is given by
$$
\mathcal{P}_{\infty}(u)=\frac{3}{2}\int_{\R^3}|\nabla u|^2dx+\frac{3}{4}\int_{\R^3}(I_{2}*|u|^2)|u|^2dx-\frac{2p-3}{p}\int_{\R^3}|u|^pdx.
$$
Similarly we can show that $\mathscr{P}_{\infty}\neq \emptyset$ and $\mathcal{P}_{\infty}$ has the same properties as $\mathcal{P}_{\lambda}$.

In this section, we are going to show that $v_\lambda$ converges to a positive ground-state of the formal limit equation \eqref{eqssps}.
Similar to the proof of \cite[Corollary 3.2]{IR12}, by using Lemma~\ref{lem1201} we obtain a lower bound on $M(u)$ for all $u\in \mathscr{P}_{\infty}$.
\begin{corollary}\label{cor-1201}
	Let $p\in (3, 6)$. 	There exists $\eta>0$ such that $M(u)>\eta$ for any $u\in \mathscr{P}_{\infty}$.
\end{corollary}
With Proposition~\ref{pro1201}, we could prove that $\mathscr{P}_{\infty}$ is a natural constraint in the spirit of \citelist{\cite{AP08}\cite{R06}} and the proof will be skiped.
\begin{lemma}\label{lem1202}
	Assume that $p\in (3, 6)$. 	Then $m_{\infty}>0$ and $\mathscr{P}_{\infty}$ is a natural constraint.
\end{lemma}
\begin{remark}\label{rem-ssps}
\cite[Theorem 1.1]{IR12} shows that \eqref{eqssps} admits a ground state $v_{\infty}\in \mathscr{P}_{\infty}$ with $\widetilde{\mathcal{I}}_{\infty}(v_{\infty})=m_{\infty}$ when $p\in (3,6)$.  Moreover, each solution of \eqref{eqssps} has an exponential decay at infinity if $p\in (3, 6)$ \cite[Theorem 1.3]{BJL13}, which means that they belong to $L^2(\R^3)$.
\end{remark}

We now turn our attention to study the asymptotic profile of the rescaled ground state $v_{\lambda}$.  As we mentioned in the introduction, we use direct variational analysis based on the comparison of the ground state energy levels for two problems, we begin by studying the convergence of $m_{\lambda}$ as $\lambda\to \infty$.

\begin{lemma}\label{lem1101}
	Assume that $p\in (3, 6)$. 	Then $0<m_{\lambda}-m_\infty\to 0$ as $\lambda\to \infty$.
\end{lemma}

\begin{proof}
	First, we use $v_\lambda$ with $\lambda>0$ as a test function for $\P_\infty$. We obtain
    $$\mathcal{P}_\infty(v_{\lambda})=\mathcal{P}_{\lambda}(v_{\lambda})-\frac{\lambda^{\frac{p-2}{4(3-p)}}}{2}\|v_{\lambda}\|_2^2=-\frac{\lambda^{\frac{p-2}{4(3-p)}}}{2}\|v_{\lambda}\|_2^2<0.$$
	Hence there exists a unique $t_{\lambda}\in (0, 1)$ such that $t^2_{\lambda}v_{\lambda}(t_{\lambda}x)\in \mathscr{P}_\infty$, and we have
	\begin{equation}\label{eq1102}
	m_\infty\leq \widetilde{\mathcal{I}}_{\infty}(t^2_{\lambda}v_{\lambda}(t_{\lambda}x))=\frac{t_{\lambda}^{3}(p-3)}{2p-3}\|\nabla v_\lambda\|_2^2+\frac{(p-3) t_{\lambda}^{3}}{2(2p-3)}\int_{\R^3}(I_{2}*|v_{\lambda}|^2)|v_{\lambda}|^2dx< \widetilde{\I}_{\lambda}(v_{\lambda}(x))=m_{\lambda},
	\end{equation}
	which means $m_\infty<m_{\lambda}$.
	
	To show that $m_{\lambda}\to m_\infty$ as $\lambda\to \infty$ we shall use $v_\infty$ as a test function for $\P_\lambda$. Note that the ground state $v_\infty$ of \eqref{eqssps} has an exponential decay at infinity, hence $v_\infty\in L^2(\R^3)$ (see Remark~\ref{rem-ssps}).
	
	Since $\mathcal{P}_{\lambda}(v_\infty)=\frac{\lambda^{\frac{p-2}{4(3-p)}}}{2}\|v_\infty\|_2^2>0$, there exists $\overline{t}_{\lambda}>1$ such that $\overline{t}^2_{\lambda}v_{\lambda}(\overline{t}_{\lambda}x)\in \mathscr{P}_{\lambda}$, i.e.,
	$$
	\frac{3\overline{t}_{\lambda}^{3}}{2}\|\nabla v_\infty\|_2^2+\frac{\lambda^{\frac{p-2}{4(3-p)}}\overline{t}_{\lambda}}{2}\|v_\infty\|^2_2+\frac{3 \overline{t}_{\eps}^{3}}{4}\int_{\R^3}(I_{\alpha}*|v_\infty|^2)|v_\infty|^2 dx=\frac{(2p-3)\overline{t}_{\lambda}^{2p-3}}{p}\|v_\infty\|_p^p.
	$$
	This, combined with $\mathcal{P}_{\infty}(v_\infty)=0$ and $v_\infty\in L^2(\R^N)$, implies that
	$$
	\frac{(2p-3)(\overline{t}_{\lambda}^{2p-6}-1)}{p}\|v_\infty\|_p^p=\frac{\lambda^{\frac{p-2}{4(3-p)}}\overline{t}^{-2}_{\lambda}}{2}\|v_\infty\|^2_2.
	$$
	Therefore, $\overline{t}_{\lambda}\to 1$ as $\lambda\to \infty$.  Moreover,
	$$
	\overline{t}_{\lambda}\leq 1+C\lambda^{\frac{p-2}{4(3-p)}},
	$$
	where $C>0$ is independent of $\lambda$.
	Thus we have
	$$\aligned
	m_{\lambda}\leq\widetilde{\mathcal{I}}_{\lambda}(t^2_{\lambda}v_{\lambda}(t_{\lambda}x))\leq&\widetilde{\mathcal{I}}_{\infty}(v_\infty)
	+C(\overline{t}_{\lambda}^{3}-1)+\frac{\lambda^{\frac{p-2}{4(3-p)}}\overline{t}_{\lambda}}{2}\|v_\infty\|^2_2\\
	\leq &m_\infty+C\lambda^{\frac{p-2}{4(3-p)}}.
	\endaligned
	$$
	This, together with \eqref{eq1102}, means that $m_{\lambda}-m_\infty\to 0$ as $\lambda\to \infty$.
	\end{proof}

\begin{corollary}\label{cor-bound0}
	Let $p\in (3, 6)$. 	Then the quantities
	$$\|\nabla v_{\lambda}\|_2^2,\quad \lambda^{\frac{p-2}{4(3-p)}}\|v_{\lambda}\|_2^2,\quad \|v_{\lambda}\|_q^q,
	\quad \int_{\R^3}(I_{2}*|v_{\lambda}|^2)|v_{\lambda}|^2 dx,$$
	are uniformly bounded as $\lambda\to \infty$.
\end{corollary}

\begin{proof}
	From $v_\lambda\in \mathscr{P}_\lambda$ and Lemma \ref{lem1101} we have
	\begin{multline*}
	m_\infty+o(1)=m_{\lambda}=\widetilde{\mathcal{I}}_{\lambda}(v_\lambda(x))\\
	=\frac{p-3}{2p-3}\|\nabla v_\lambda\|_2^2+\frac{p-2}{2p-3}\lambda^{\frac{p-2}{4(3-p)}}\| v_\lambda\|_2^2+\frac{p-3}{2(2p-3)}\int_{\R^3}(I_{2}*|v_\lambda|^2)|v_\lambda|^2dx.
	\end{multline*}
Therefore, $$\|\nabla v_{\lambda}\|_2^2,\quad \lambda^{\frac{p-2}{4(3-p)}}\|v_{\lambda}\|_2^2,\quad \|v_{\lambda}\|_q^q,
	\quad \int_{\R^3}(I_{2}*|v_{\lambda}|^2)|v_{\lambda}|^2 dx,$$
	are uniformly bounded as $\lambda\to \infty$.
	\end{proof}

\begin{lemma}\label{lem1102}
	Let $p\in (3, 6)$.	Then $\lambda^{\frac{p-2}{4(3-p)}}\|v_{\lambda}\|_2^2\to 0$ as $\lambda\to \infty$.
\end{lemma}

\begin{proof}
	Lemma~\ref{lem1101} implies that there exists a unique $t_{\lambda}\in (0, 1)$ such that $t_{\lambda}^2v_{\lambda}(t_{\lambda}x)\in \mathscr{P}_\infty$.  Indeed, assume that $t_{\lambda}\to t_\infty<1$ as $\lambda\to \infty$. Then by \eqref{eq1102} we have, as $\lambda\to \infty$,
	\begin{equation}\label{teto0}\aligned
	m_\infty &\leq \widetilde{\mathcal{I}}_{\infty}(t^2_{\lambda}v_{\lambda}(t_{\lambda}x))=\frac{t_{\lambda}^{3}(p-3)}{2p-3}\|\nabla v_\lambda\|_2^2+\frac{t_{\lambda}^{3}(p-3)}{2(2p-3)}\int_{\R^3}(I_{2}*|v_{\lambda}|^2)|v_{\lambda}|^2dx\\
    &< t_{\lambda}^{3}\widetilde{\I}_{\lambda}(v_{\lambda}(x))=t_{\lambda}^{3}m_{\lambda}
	\to t_\infty^{3} m_\infty<m_\infty,
    \endaligned	
    \end{equation}
	which is a contradiction.  Therefore $t_{\lambda}\to 1$ as $\lambda\to \infty$. Using $\mathcal{P}_\infty(t_{\lambda}^2v_{\lambda}(t_{\lambda}x))=0$ again, we see that
	$$\aligned
	0=&\frac{3t_{\lambda}^{3}}{2}\|\nabla v_{\lambda}\|_2^2+
	\frac{3t_{\lambda}^3}{4}\int_{\R^3}(I_{2}*|v_{\lambda}|^2)|v_{\lambda}|^2dx-\frac{(2p-3)t^{2p-3}_{\lambda}}{p}\|v_{\lambda}\|_p^p\\
	=&\mathcal{P}_{\lambda}(v_{\lambda})-\frac{\lambda^{\frac{p-2}{4(3-p)}} t_{\lambda}}{2}\|v_{\lambda}\|_2^2+\frac{3(t^{3}_{\lambda}-1)}{2}\|\nabla v_{\lambda}\|_2^2+\frac{3(t^{3}_{\lambda}-1)}{4}\int_{\R^3}(I_{2}*|v_{\lambda}|^2)|v_{\lambda}|^2dx\\
	&-\frac{(2p-3)(t^{2p-3}_{\lambda}-1)}{p}\|v_{\lambda}\|_p^p.
	\endaligned
	$$
	This, together with Corollary~\ref{cor-bound0} and $\mathcal{P}_{\lambda}(v_{\lambda})=0$, implies that $\lambda^{\frac{p-2}{4(3-p)}}\|v_{\lambda}\|_2^2\to 0$ as $\lambda\to \infty$.
\end{proof}

\begin{proof}[Proof of Theorem~\ref{thm01}]

	From Corollary \ref{cor-bound0}, Lemma~\ref{lem1102} and \eqref{teto0}, we see that $\{t_{\lambda}^2v_{\lambda}(t_{\lambda}x)\}\subset \mathscr{P}_\infty$ is a minimizing sequence for $m_\infty$ which is bounded in $E(\R^3)$ and $t_{\lambda}\to 1$ as $\lambda\to \infty$.  For any sequence $\lambda_n\to\infty$, up to subsequence, it follows from Proposition~\ref{pro1203} that $\{t_{\lambda_n}^2v_{\lambda_n}(t_{\lambda_n}x)\}$ converges weakly in $E(\R^3)$.
Applying the Concentration-Compactness principle [cf. \cite[Chapter I, Theorem 4.9]{S96}] to $t_{\lambda}^2v_{\lambda}(t_{\lambda}x)$ and following the proof of \cite{AP08}, we conclude that there exists $\{\xi_{\lambda_n}\}\subset \R^3$ and $\overline{v}_\infty\in E(\R^3)$ such that
$$t_{\lambda_n}^2v_{\lambda_n}(t_{\lambda_n}(x-\xi_{\lambda_n}))\to \overline{v}_\infty \quad \text{in}\ L^{q}(\R^3), \ q\in (3,6).$$
By using nonlocal Brezis-Lieb Lemma \cite{MMS16}*{Proposition 4.7}, we obtain that
$$
\lim_{\lambda_n\to\infty}\int_{\R^3}(I_{2}*|t_{\lambda_n}^2v_{\lambda_n}(t_{\lambda_n}x)|^2)|t_{\lambda_n}^2v_{\lambda_n}(t_{\lambda_n}x)|^2dx=\int_{\R^3}(I_{2}*|\overline{v}_\infty|^2)|\overline{v}_\infty|^2dx,
$$
which means that $\overline{v}_\infty\neq0$ since by Corollary~\ref{cor-1201} the sequence $\{M(t_{\lambda}^2v_{\lambda}(t_{\lambda}x))\}$ has a positive lower bound.  It follows from the weakly lower semi-continuity of the norm that $0=\mathcal{P}_{\infty}(t_{\lambda}^2v_{\lambda}(t_{\lambda}x))\geq \mathcal{P}_{\infty}(\overline{v}_\infty)$, then there exists a unique $t_{\infty}\in (0, 1]$ such that $t_{\infty}^2\overline{v}_{\infty}(t_{\infty}x)\in \mathscr{P}_\infty$.  Note that $\{t_{\lambda_n}^2v_{\lambda_n}(t_{\lambda_n}(x-\xi_{\lambda_n}))\}$ is also a bounded minimizing sequence for $m_\infty$, we have
$$\aligned
m_\infty+o(1)&=\widetilde{\mathcal{I}}_{\infty}(t_{\lambda_n}^2v_{\lambda_n}(t_{\lambda_n}(x-\xi_{\lambda_n})))\\
&= \frac{t_{\lambda_n}^{3}(p-3)}{2p-3}\|\nabla v_{\lambda_n}\|_2^2+\frac{t_{\lambda_n}^{3}(p-3)}{2(2p-3)}\int_{\R^3}(I_{2}*|v_{\lambda_n}|^2)|v_{\lambda_n}|^2dx \\
&\geq \frac{(p-3)}{2p-3}\|\nabla v_\infty\|_2^2+\frac{(p-3)}{2(2p-3)}\int_{\R^3}(I_{2}*|v_{\infty}|^2)|v_{\infty}|^2dx\\
&\geq \frac{t_{\infty}^3(p-3)}{2p-3}\|\nabla v_\infty\|_2^2+\frac{t_{\infty}^3(p-3)}{2(2p-3)}\int_{\R^3}(I_{2}*|v_{\infty}|^2)|v_{\infty}|^2dx\\
&=\widetilde{\mathcal{I}}_{\infty}(t_{\infty}^2\overline{v}_{\infty}(t_{\infty}x))\\
&\geq m_{\infty},
\endaligned
$$
which implies that $t_{\infty}=1$ and $t_{\lambda_n}^2v_{\lambda_n}(t_{\lambda_n}(x-\xi_{\lambda_n}))\to \overline{v}_\infty \ \text{in}\ D^{1}(\R^3)$, therefore $\overline{v}_{\infty}\in \mathscr{P}_\infty$, $\widetilde{\mathcal{I}}_{\infty}(\overline{v}_{\infty})=m_{\infty}$ and
$v_{\lambda_n}(\cdot-\xi_{\lambda_n})\to \overline{v}_\infty \ \text{in}\ E(\R^3)$ since $t_{\lambda_n}\to 1$ as $n\to\infty$. Hence, by Lemma~\ref{lem1202}, $\overline{v}_\infty$ is a ground state solution of \eqref{eqssps}.
\end{proof}

\vspace{5pt}
\noindent{\bf Acknowledgements.}
The first author is partially supported by Natural Science Foundation of China (Nos. 11901418, 11771319, 12171470).

\end{document}